\documentclass{amsart}
\usepackage{amsfonts,amssymb,amsmath,amsthm}
\usepackage{url}
\usepackage{enumerate}
\usepackage{hyperref}
\usepackage{fancyhdr}
\newtheorem{theorem}{Theorem}[section]
\newtheorem{lemma}[theorem]{Lemma}

\newtheorem{corollary}[theorem]{Corollary}

\theoremstyle{definition}

\newtheorem{hypothesis}[theorem]{Hypothesis}

\newtheorem{remarks}[theorem]{Remarks}

\def\Q{\mathbb{Q}}
\def\Z{\mathbb{Z}}

\def\lien{\mathrel{\mkern-4mu}}
\def\too{\relbar\lien\rightarrow}
\def\tooo{\relbar\lien\relbar\lien\too}
\def\No{{\rm N}}
\def\Cl{{\mathcal C}\hskip-2pt{\ell}}
\def\cl{c\hskip-1pt{\ell}}

\def\ov{\overline}
\def\Frac#1#2{\hbox{\footnotesize $\displaystyle \frac{#1}{#2}$}}
\def\plus{\displaystyle\mathop{\raise 2.0pt \hbox{$\bigoplus $}}\limits}
\def\prd{ \displaystyle\mathop{\raise 2.0pt \hbox{$\prod$}}\limits}
\def\sm{  \displaystyle\mathop{\raise 2.0pt \hbox{$\sum$}}\limits}

\author[Georges  Gras]{Georges  Gras }
\address{Villa la Gardette \\ chemin Ch\^ateau Gagni\`ere \\  F--38520 Le Bourg d'Oisans.}
\email{g.mn.gras@wanadoo.fr \   {\it url}:\url{http://www.researchgate.net/profile/Georges_Gras}}


\begin{document}

\title[No Riemann-Hurwitz formula for relative $p$-class groups ]
{No Riemann-Hurwitz formula for the $p$-ranks \\ of relative class groups}
 
\begin{abstract} 
We disprove, by means of numerical examples, the existence of a Riemann-Hurwitz formula for the $p$-ranks
of relative class groups in a $p$-ramified $p$-extension $K/k$ of number fields of CM-type containing $\mu_p$.
In the cyclic case of degree~$p$, under some assumptions on the $p$-class group of $k$, we prove some properties of the 
Galois structure of the $p$-class group of $K$;
but we have found, through numerical experimentation, that some theoretical group structures do not exist in this particular
situation, and we justify this fact. Then we show, in this context, that Kida's formula on lambda invariants is valid for the $p$-ranks 
if and only if the $p$-class group of $K$ is reduced to the group of ambiguous classes, which is of course not always the case.
\end{abstract}

\maketitle

\vspace{-0.1cm}
\centerline{December  14, 2015}

\section{Generalities}\label{notations}

In 1980, Y. Kida \cite{Ki} proved an analogue of the Riemann-Hurwitz formula for
the minus part of the Iwasawa $\lambda$-invariant of the cyclotomic $\Z_p$-extension $k_\infty$
of an algebraic number field $k$ of CM-type with maximal totally real subfield $k^+$:
$$\lambda^-(K) -1= [K_\infty : k_\infty] \ (\lambda^-(k) -1) + \sm_{v \nmid p} (e_v(K^+_\infty / k^+_\infty) - 1), $$
where $K/k$ is a finite $p$-extension of CM-fields containing the group $\mu_p$ of $p$th roots of unity and where $e_v$ 
denotes the ramification index of the place $v$. When $K/k$ is $p$-ramified (i.e., unramified outside $p$) and such that
$K \cap k_\infty =k$ the formula reduces to:
$$\lambda^-(K) -1= [K : k] \, (\lambda^-(k) -1). $$

\smallskip\noindent
Many generalizations where given as in \cite{JaMa}, \cite{JaMi}, \cite{Sch}, among many others. 

\smallskip
An interesting question is to ask if such a formula can be valid for the $p$-ranks of the relative ideal class groups in 
a $p$-extension $K/k$ (e.g. $K/k$ cyclic of degree $p$).

\smallskip
In a work,  using Iwasawa theory and published in 1996 by K. Wingberg \cite{W}, such a formula is proposed 
(Theorem 2.1, Corollary 2.2) in a very general framwork and applied to the case of a $p$-ramified 
$p$-extension $K/k$ of CM-fields containing $\mu_p$. 

\smallskip
As many people remarked, we can be astonished by a result which is not really  ``arithmetical'' since many of our class groups investigations (as in \cite{Gr2}, \cite{Gr3}) show that such a ``regularity'' only happens at infinity (Iwasawa theory).
The proof of Kida's formula given by W. Sinnott \cite{Sin}, using $p$-adic $L$-functions, is probably the most appropriate to see 
the transition from one aspect to the other.

\smallskip
 Indeed, to find a relation between the $p$-ranks (for instance in a cyclic extension $K/k$ of degree~$p$) depends
 for $G:= {\rm Gal}(K/k) =: \langle \sigma \rangle$
on non obvious structures of  finite $\Z_p[G]$-modules $M$ provided with an arithmetical norm $\No_{K/k}$ and a transfert 
map $j_{K/k}$ (with $j_{K/k} \circ \No_{K/k} = 1+ \sigma + \ldots + \sigma^{p-1}$), where the filtration of the $M_i := \{h \in M,\ h^{(\sigma -1)^i} = 1\}$ plays an important non-algebraic role because all the orders $\# (M_{i+1}/M_i)$ depend on {\it arithmetical local normic computations} by means of formulas, given in \cite{Gr2}, similar to that of the case $i=0$ of ambiguous ideal classes 
(see Section \ref{structure}).

\smallskip
To be more convincing, we have given numerical computations and we shall see that it is not difficult
to conjecture that there are infinitely many counterexamples to a formula, for the relative $p$-ranks, 
such as $r_K^- - 1 = p\,(r_k^- -1)$ wich may be true in some cases.

\section{A numerical counterexample}\label{contre}

Using PARI (from \cite{P}), we give a program which can be used by the reader to compute easily for the case $p=3$
in a biquadratique field containing a primitive $3$th root of unity $j$, and such that its $3$-class group is of order 3.

\subsection{Definition and assumptions}\label{def}
Consider the following diagram:
\unitlength=0.86cm

$$\vbox{\hbox{\hspace{-1.0cm}  
\begin{picture}(11.5,7.2)
\put(4.35,0.6){\line(1,0){2.2}}
\put(5.1,4.0){\line(1,0){1.5}}
\put(4.6,7.06){\line(1,0){2.0}}
\put(4.0,1.0){\line(0,1){2.6}}
\put(7.1,1.0){\line(0,1){2.6}}
\put(4.0,4.3){\line(0,1){2.6}}
\put(7.1,4.3){\line(0,1){2.6}}
\bezier{150}(4.3,1.0)(4.75,1.45)(5.2,1.9)
\bezier{150}(5.8,2.6)(6.25,3.1)(6.7,3.6)
\put(3.8, 0.5){$\Q$}
\put(2.4, 3.8){$k^+\!=\!\Q(\sqrt{3 . d\,})$}
\put(4.5, 2.1){$k^-\!=\!\Q(\sqrt{- d})$}
\put(6.8, 0.5){$\Q(j)$}
\put(6.8, 3.84){$k\!=\!k^+(j)$}
\put(3.7, 7.0){$K^+$}
\put(6.7, 7.0){$K\!=\! k(\sqrt[3] {\alpha\,})$}
\put(7.2, 5.4){$[K:k]\!=\!3$}
\put(5.4, 7.2){$s$}
\end{picture}   }} $$

We recall, in this particular context, the hypothesis of the statement of  Corollary 2.2 (ii) of \cite{W}
and we shall suppose these conditions satisfied in all the sequel.
For any field $F$, let $\Cl_F$ be its 3-class group and let
$\Cl_F^{\pm}$ be its two usual components when $F$ is a CM-field. We denote by $H_F$ the $3$-Hilbert class field of $F$.

\begin{hypothesis}\label{hypo}
(i) $d >0$, $d$ squarefree, $d \not \equiv 0 \pmod 3$,

(ii) $p=3$ does not split in $k/\Q$ (hence $d \equiv 1 \pmod 3$),

(iii) $K^+/k^+$ is a $3$-ramified cubic cyclic extension and $K=K^+(j)$,

(iv) $K^+$ is not contained in the cyclotomic $\Z_3$-extension of $k^+$,

(v) $\Cl_{k^+} = 1 \ \& \ \Cl_{k^-} \simeq \Z/3\Z$.
\end{hypothesis}

 From (v), the ambiguous class number formula implies $\Cl_{K^+}=1$ (Lemma \ref{lemmA}).

\smallskip
Starting from a $p$-ramified cubic cyclic extension $K^+/k^+$, the associated Kummer extension $K/k$ is defined by 
$\sqrt[3]\alpha$ such that $\alpha \in k^\times \setminus k^{\times 3}$, $(\alpha) = {\mathfrak a}^3$ for an ideal 
${\mathfrak a}$ of $k$, and $\alpha^{s+1} \in  k^{\times 3}$ where $s \in {\rm Gal}(K/K^+)$ is the complex conjugation 
(usual decomposition criterion of a Kummer extension over a subfield). 

\smallskip
If the $3$-rank of the $3$-class group $\Cl_{k}^- \simeq \Cl_{k^-}$ is $r^-$, the $3$-rank of the Galois group of the maximal 
Abelian $3$-ramified $3$-extension of $k^+$ is $r^- + 1$ when $3$ does not split in $k$ (see \cite{Gr1}, 
Proposition III.4.2.2 for the general statement). From (iv), necessarily $r^- \geq 1$.
Here we suppose $\# \Cl_{k^-} =3$, hence $r^- + 1=2$  (this is also equivalent to the non-nullity of the 3-torsion subgroup 
${\mathcal T}_3$, of the above Galois group, whose order is in our context $\# {\mathcal T}_3 \sim 
{\rm log}_3(\varepsilon)/ \sqrt{3. d}$ where $\varepsilon$ is the fundamental unit of $k^+$, see \cite{Gr1}, Remark III.2.6.5 (i)). 
We shall precise the choice of the extension $K^+/k^+$ (i.e., $\alpha$) as follows:

\smallskip
Let $\alpha$ in $k^-$ be such that $(\alpha) = {\mathfrak a}^3$ (${\mathfrak a}$ must be 
a non-principal ideal since $E_{k^-}$ is trivial); this defines a canonical cubic cyclic extension $K^+$ of $k^+$ and the numbers 
$\alpha\,j$, $\alpha\,j^2$ define the two other cubic cyclic extensions $K^+_1$, $K^+_2$ of $k^+$, distinct from the first step
of the cyclotomic $\Z_3$-extension $K^+_0$ of $k^+$ defined by $j$.

\begin{lemma} \label{lemmA}
(i) We have  $\#\Cl_{K^+} = 1 \ \  \& \ \  \# \Cl_{K}^G =\# \Cl_{k}= 3$, where $G = {\rm Gal}(K/k)$. 

\smallskip
(ii) We have $\#\Cl_{H_k} = 1$.
\end{lemma}

\begin{proof} (i) Using the Chevalley's formula in $K^+/k^+$ (see e.g. \cite{Gr1}, Lemma II.6.1.2) 
with a trivial $3$-class group for $k^+$, the formula reduces to 
$$\# \Cl_{K^+}^G = \Frac{3}{3\,.\, (E_{k^+} : E_{k^+} \cap\No_{K^+/k^+}K^+{^\times} )} = 1, $$
since the product of ramification indices is equal to 3 ($\Cl_{H_{k} }=1$ implies that $K^+/k^+$ is necessarily totally ramified
at the single prime above 3 of $k^+$).

\smallskip
The same formula in $K/k$ is $\# \Cl_{K}^G = \Frac{\#\Cl_{k} \times 3}{3\,.\,  (E_{k} : E_{k} \cap\No_{K/k}K{^\times} )}$ with 
$E_k = \langle \varepsilon, j \rangle$, where $\varepsilon$ is the fundamental unit of $k^+$;
but, as for $K^+/k^+$, there is by assumption a single place of $k$ ramified in $K/k$, thus, using 
the product formula, the Hasse norm theorem shows that all these units are local norms everywhere hence global norms.
So $\# \Cl_{K}^G =3$ since $\#\Cl_{k} = \#\Cl_{k^-}=3$.

\smallskip
(ii) We have $\# \Cl_{H_k}^{G'} = \Frac{\#\Cl_{k} }{3\,.\,  (E_{k} : E_{k} \cap\No_{H_k/k}H_k^{\times} )}=1$
where $G' := {\rm Gal}(H_k/k)$.
\end{proof}

\begin{corollary}\label{coroB}
 We have $\Cl_{K}= \Cl_{K}^-$ since $\Cl_{K}^+ \simeq \Cl_{K^+} =1$.
\end{corollary}

\subsection{Numerical data} We have a first example with $d=211 \equiv 1 \pmod 3$ and $\alpha = \frac{17 +  \sqrt {-211}}{2}$
where $(\alpha) = {\mathfrak p}_5^3$ for a non-principal prime ideal dividing $5$ in $k^- $. 

\smallskip
The class number of $k^+$ is $1$ and that of $k^-$ is $3$, which is coherent with the fact that the fundamental unit 
$\varepsilon = 440772247+17519124\,\sqrt {3 . 211}$ of $k^+$ is $3$-primary (indeed, $\varepsilon
\equiv 1 +3 .\,\sqrt {3 . 211} \pmod 9$), which implies that $H_{k^-}$ is given via $k(\sqrt[3] {\varepsilon})/k$
which is decomposed by means of an unramified cubic cyclic extension of $k^-$.

\smallskip
So all the five conditions (i) to (v) are fulfilled.

\smallskip
The PARI program (see Section \ref{data}) gives in ``$component(H,5)$'' the class number and the structure of the whole 
class group of $K$; the program needs an irreducible polynomial defining $K$;  it is given by ``$P=polcompositum(x^2+x+1, Q)$'' where $Q=x^6-17\,x^3+5^3$ is the irreducible polynomial 
of $\sqrt[3] \alpha$ over $\Q$ (the general formula is $Q=x^6-{\rm Tr}_{k^-/\Q}(\alpha)\,x^3+\No_{k^-/\Q} (\alpha)$);
one obtains:

\smallskip\noindent
$P=x^{12} - 6\,x^{11} + 21\,x^{10} - 84\,x^9 + 243\,x^8 - 432\,x^7 + 1037\,x^6 - 1896\,x^5 - 204\,x^4 - 966\,x^3 + 5949\,x^2 + 4905\,x + 11881$.

\subsection{Conclusion} The program gives $\Cl_{K} \simeq \Z/9\,\Z \times \Z/3\,\Z$ for a class number 
equal to $27$. This yields the 3-rank $R^- = 2$ of $\Cl_{K}^-$ when the 3-rank $r^-$ of 
$\Cl_{k}^-$ is equal to~1, which is incompatible with the formula
$$R^--1 = 3\times (r^--1). $$

But, this ``Riemann-Hurwitz formula'' is valid if and only if $\Cl_{K}= \Cl_{K}^G \simeq \Z/3\,\Z$
(no exceptional 3-classes). Such a case is also very frequent (see \S\,\ref{inv}).

\section{Some structural results}\label{structure}
Denote by $M$ a finite $\Z_p[\Gamma]$-module, where we assume that $\Gamma$ is an Abelian Galois group
of the form  $G \times g$, where $G= {\rm Gal} (K/k) =: \langle \sigma \rangle$ is cyclic of order $p$ and
$g \simeq {\rm Gal} (k/k_0)$  (of order prime to $p$), where $k_0$ is a suitable subfield of $k$ (so that $K=k\,K_0$
with $K_0:= K^g$). The existence of $g$ allows us to take isotypic components of $M$ (as the $\pm$-components 
when the fields are of CM-type).  In our example, $g=\langle s \rangle$, $k_0=k^+$ and $K_0=K^+$.

\smallskip
We have $M/M^p = M/M^{1+\sigma+\ldots +\sigma^{p-1}-\Omega}$
where $\Omega = u\, (\sigma -1)^{p-1}$ for an inversible element $u$ of the group algebra $\Z_p[G]$ 
(\cite{Gr3}, Proposition 4.1); in our case $p=3$, $\Omega =\sigma^2 (\sigma-1)^2$. 
We shall use:
$$\omega :=\sigma \, (\sigma-1), \ \ \hbox{such that $\omega^2 \equiv 3\!\! \pmod {\nu}$, \  where 
$\nu := 1+ \sigma + \sigma^2$} .$$

\smallskip
For our purpose we shall have $M=\Cl_{K}^-$ (we refere to \cite{Gr3}, Chap. IV, A, \S\,2).

\smallskip
By class field theory, when $K/k$ is totally ramified at  the unique ${\mathfrak p} \mid 3$, 
the arithmetical norm $\No_{K/k} : \Cl_{K}^- \too \Cl_{k}^-$ is surjective.

\smallskip
Another important fact  for the structure of $\Cl_K$, in our particular context, is that the class of order 3 of $k$ capitulates in $K$
because the equality $(\alpha) = {\mathfrak a}^3$ becomes $(\sqrt[3] \alpha) = ({\mathfrak a})_K$ in $K$ 
(the transfert map $j_{K/k} : \Cl_k^- \too \Cl_K^-$ is not injective).
This has the following tricky consequence:
$$\No_{K/k} (\Cl_K^-) = \Cl_k^- \ \ \&  \ \  (\Cl_K^-)^\nu = 1. $$

Return to the general case $M= \Cl_K$ for any prime $p$ and suppose $\# M^G = p$.
Put $M_i := \{h \in M,\   h^{(\sigma-1)^i}=1\}$, $i \geq 0$, and let $n$
 be the least integer $i$ such that $M_i = M$.

 From the exact sequence
$1 \to M_1=M^G \too M_{i+1} \displaystyle \mathop{\tooo}^{\omega}  M_{i+1}^\omega \subseteq M_i \to 1$,
with $\# M^G=p$, we obtain that $M_{i+1}^\omega = M_i$ and $\# (M_{i+1}/M_i)=p$ for 
$i=0, \ldots, n-1$. 
From \cite{Gr3}, Proposition 4.1 and Corollaire 4.3, assuming  $M^\nu=1$ and $\#(M_{i+1}/M_i) = p$ for all $i < n$, 
we obtain the following structure of $\Z$-module:
$$M \simeq (\Z/p^{a+1}\Z)^b \times (\Z/p^a\Z)^{p-1-b}, $$
where $n=a\,(p-1) + b$, $0 \leq b \leq p-2$. This implies (assuming $M^+=1$) that the $p$-rank $R^-$
of $M^-$ is $R^-= p-1$ if $a\geq 1$  and $R^-= b$ if $a=0$ (i.e., $b=n \leq p-2$). 

\smallskip
So, in the case 
$r^- =1$, we have the Riemann-Hurwitz formula $R^- -1=p\,(r^- -1)$ if and only if $b=n=1$
which is equivalent to $M^- = M^-{}^G$. Otherwise, $R^-$ can take any value in $[1, p-1]$, even if $r^-=1$.

\smallskip
This general isomorphism comes from the formula (\cite{Gr2}, Corollaire 2.8):
$$\# \big(M_{i+1}/M_i \big)^G = \frac{\# \Cl_{k} \times \prd e_v }{ [K : k] \, .\,\#\No_{K/k}(M_i) 
\, .\, (\Lambda_i : \Lambda_i  \cap \No_{K/k}(K^\times) )} , $$
where $\No_{K/k}(M_i) := \cl_k(\No_{K/k}(I_i) )$ for a suitable ideal group $I_i$ such that
$\cl_K(I_i)=M_i$, and where $\Lambda_i = \{x \in k^\times, \  (x) \in \No_{K/k}(I_i) \}$.

\smallskip
In our case there is a single ramified place ${\mathfrak p} \, \vert \, 3$ and the elements of $\Lambda_i$, being norms of ideals, 
are everywhere local norms except perhaps at ${\mathfrak p}$; so $(\Lambda_i : \Lambda_i  \cap \No_{K/k}(K^\times) )=1$
under the product formula of class field theory and the Hasse norm theorem, and 
$\# \big(M_{i+1}/M_i \big)^G = \Frac{3 }{ \#\No_{K/k}(M_i)}$.

\smallskip
In our numerical example with $d=211$, we get necessarily $a=b=1$, $n=3$, giving the structure $\Z/9\Z \times \Z/3\Z$.

\medskip
For more structural results when $M^\nu\ne 1$, see \cite{Gr3}, Chap. IV, \S\,2, Proposition 4.3 
valid for any $p \geq 2$. In our biquadratic case and $p=3$, we obtain interesting structures for which 
a theoretical study should be improved.
From the general PARI program we have obtained the following numerical examples:

\smallskip
(i) For $d=1759$, for which $\#\Cl_{k^+}=1$, $\Cl_{k^-} \simeq \Z/27\Z$, and $\alpha = 37+ 20\,\sqrt{-d}$
of norm $89^3$,
the structure is $\Cl_{K}^- \simeq \Z/27\Z$ (i.e., $\Cl_{K}^-=(\Cl_{K}^-)^G$).

\smallskip
(ii) For $d=2047$,  for which $\#\Cl_{k^+}=1$, $\Cl_{k^-} \simeq \Z/9\Z$, and $\alpha= 332+ 11\,\sqrt{-d}$
of norm $71^3$, the structure  is
$\Cl_{K}^- \simeq \Z/9\Z \times \Z/3\Z \times \Z/3\Z$.

\smallskip
(iii) For $d=1579$, for which $\#\Cl_{k^+}=1$, $\Cl_{k^-} \simeq \Z/9\Z$, and $\alpha = \frac{1}{2} (115+ 3\,\sqrt{-d})$ of norm $19^3$,
the structure is $\Cl_{K}^- \simeq \Z/27\Z \times \Z/9\Z \times \Z/3\Z$.

\section{The structure $M \simeq \Z/3^a\Z \times \Z/3^a\Z$ does not exist}

Of course, we keep the same numerical assumptions about the $3$-class groups of $k^+$ and $k^-$ (especially 
$\Cl_{k^-} \simeq \Z/3\Z$), the non-splitting of $3$ in $k/\Q$, and the Kummer construction of 
the $3$-ramified cubic cyclic extension $K/k$.

\begin{theorem}\label{thmf} Under all the Hypothesis \ref{hypo}, we get
$$\Cl_K = \Cl_K^- \simeq \Z/3^{a+1}\Z \times \Z/3^a\Z, $$ 
for some $a \geq 0$. The case $a=0$ is equivalent to $\Cl_K=\Cl_K^G$.
\end{theorem}

\begin{proof} Let $n\geq 1$ be the least integer such that $M_n = M := \Cl_K$. From the relations
$M_{i+1}/M_i \simeq \Z/3\Z$, for $0\leq i \leq n-1$, we get $\# M= 3^n$.

\smallskip
Let $h_n \in M$ be such that $\No_{K/k}(h_n)$ generates $\Cl_k$ (equivalent to $h_n \in M_n \setminus M_{n-1}$).
Since $M_{i+1}^\omega = M_i$ for $0\leq i \leq n-1$, $M_n$ is the $\Z_p[G]$-module generated by $h_n$ 
and for all $i$, $0\leq i \leq n-1$, $h_{n-i} := h_n^{\omega^i}$ generates $M_{n-i}$.
We have $h^{\omega^2} = h^3$ for all $h\in M$.

\smallskip
Let $m$, $1 \leq m \leq n$; the structure of $M_m$ is given (as for $M$, see Section \ref{structure}) by:
$$M_m \simeq (\Z/3^{a_m+1}\Z)^{b_m} \times (\Z/3^{a_m}\Z)^{2-b_m} , \ \, m=2\,a_m + b_m,\, b_m \in \{0, 1\} . $$

(i) Case $m=2\,e$. Then $M_m \simeq \Z/3^e\Z \times\Z/3^e\Z$.

\smallskip
(ii) Case $m=2\,e+1$. Then $M_m \simeq\Z/3^{e+1}\Z \times\Z/3^e\Z$.

\smallskip
With the previous notations we have, for the two cases:
$$M_m = \langle h_m \rangle \oplus \langle h_{m-1} \rangle,\ \ \hbox{as $\Z$-module}. $$
Indeed, suppose that $h=h_m^A = h_{m-1}^B$, $A, B \in \Z$. Then $h_m^{A-B \,\omega} =1$. Since $h_m$
is anihilated by $\omega^m$ and not by $\omega^{m-1}$, we get $A-B\,\omega \in (\omega^m)$.

\smallskip
In the case $m=2e$, $\omega^m \equiv 3^e \pmod \nu$ giving in the algebra $\Z_3[G]=\Z_3[\omega]$, the relations
$A \equiv B \equiv 0 \pmod {3^e}$, hence $h=1$. Since $\# M_{2e} = 3^{2e}$, the elements $h_m$ and 
$h_{m-1}$ are independent of order $3^e$.

\smallskip
In the case $m=2e+1$, $\omega^m \equiv 3^e\omega \pmod \nu$ and $A-B \,\omega \equiv 0 \pmod {3^e\omega}$,
giving $A = 3^e\,A'$ and $B= 3^e\,B'$, then $A' - B' \omega \equiv 0 \pmod {\omega\, \Z_3[\omega]}$, which implies
$A' \equiv 0 \pmod 3$ hence the result in this case with $h_m$ of order $3^{e+1}$ and $h_{m-1}$ of order $3^e$.

\medskip
Suppose that the structure of $M$ is $M_{2e} \simeq \Z/3^e\Z \times\Z/3^e\Z$, $e \geq 1$, and let $F$ be the subfield of  $H_K$ 
fixed by $M_{2(e-1)} = M^3$; so $F$ is a cubic cyclic extension of $K H_k$ and ${\rm Gal}(F/K) \simeq \Z/3\Z \times \Z/3\Z$.

\smallskip
The $3$-extension $F/k$ is Galois: indeed, if $\sigma$ is a generator of ${\rm Gal}(K/k)$, the action of $\sigma$
on  ${\rm Gal}(F/K) = {\rm Gal}(H_K/K)/ M^3$ is given, via the correspondence of class field theory, by the action of 
$\sigma$ on $\ov h_{2e} := h_{2e} M^3$ and on $\ov h_{2e-1} := h_{2e-1} M^3$. 
From $\omega = \sigma \,(\sigma-1) = \sigma^{-2} \,(\sigma-1)$, we have
$\ov h_{2e}^\sigma = \ov h_{2e} \times \ov h_{2e-1}^{\sigma^2}$, 
and $\sigma$ acts on $\ov h_{2e-1}$ by $\ov h_{2e-1}^\sigma =  h_{2e-1}^\sigma M^3$; 
but $h_{2e-1}^{\sigma(\sigma - 1)} = h_{2(e-1)} \in M^3$, so $h_{2e-1}^{\sigma - 1} \in M^3$ and 
$\ov h_{2e-1}^\sigma = \ov  h_{2e-1}$, hence the result and the fact that $F/H_k$ is Galois. But a group of order 
$p^2$ ($p$ prime) is Abelian and $F/H_k$ is the direct compositum of $KH_k$ and $L$ such that $L$ is the 
decomposition field at $3$ giving a cyclic extension $L/H_k$, unramified  of degree 3. 
But we know (Lemma \ref{lemmA}) that the $3$-class group of $H_k$ is trivial (contradiction).
\end{proof}

\begin{remarks} (i) Since $3$ is non-split in $k/\Q$, the unique prime ideal ${\mathfrak P} \, \vert \, 3$ in $K$ is principal:
indeed, ${\mathfrak P}^{1+s} = {\mathfrak P}^2$ gives the square of the extension of ${\mathfrak P}^+\, \vert \,  3$ 
in $K^+$ which is $3$-principal (Lemma \ref{lemmA}); so $\cl_K({\mathfrak P}) = 1$. By class field theory, ${\mathfrak P}$
splits completely in $H_K/K$.

\smallskip
(ii) In the case $n$ even for the above reasoning, an analog of the field $F$ does not exist as Galois field over $H_k$.

\smallskip
(iii) Note that the parameter $a$ can probably take any value; we have for instance obtained the following example:

\smallskip
 For $d=12058$, for which $\#\Cl_{k^+}=2$, $\Cl_{k^-} \simeq \Z/3\Z$, and $\alpha = 989+ 26\,\sqrt{-d}$ of norm $209^3$,
the structure is $\Cl_{K}^- \simeq \Z/3^4\Z \times \Z/3^3\Z$. A polynomial defining $K$ is:

\smallskip\noindent
$P = x^{12} - 6\, x^{11} + 21\, x^{10} - 4006\, x^9 + 17892\, x^8 - 35730\, x^7 + 22212821\, x^6 - 66531354\, x^5 - 113482743\, x^4 - 35777798264\, x^3 + 54059937672\, x^2 + 54106942656\, x + 83308554531904$.

\smallskip
 For $d=86942$, for which $\#\Cl_{k^+}=4$, $\Cl_{k^-} \simeq \Z/3\Z$, and $\alpha = 557+ 3\,\sqrt{-d}$ of norm $103^3$,
the structure is $\Cl_{K}^- \simeq \Z/3^5\Z \times \Z/3^4\Z$. A polynomial defining $K$ is:

\smallskip\noindent
$P =  x^{12} - 6\, x^{11} + 21\, x^{10} - 2278\, x^9 + 10116\, x^8 - 20178\, x^7 + 3449985\, x^6 - 10289502\, x^5 - 8954865\, x^4 - 2399550304\, x^3 + 3642928674\, x^2 + 3641624304\, x + 1191621124996$.
\end{remarks}

In conclusion, the structure of $\Cl_K^-$ strongly depends on the hypothesis on the order of $\Cl_k^-$
and not only of its $3$-rank.

\section{Numerical results}\label{data}
We give explicit numerical computations of $\Cl_K$, for various biquadratics fields $k$ satisfying the conditions (i) to (v)
assumed in the Hypothesis \ref{hypo}.

\smallskip
The following PARI program gives in ``$component(H,5)$'' the class number and the structure of the whole class 
group $\Cl_K$ of $K$ in the form 
$$ [class number, [c_1,  \ldots, c_\lambda] ]$$ 
such that the class group is isomorphic to $\plus_{i=1}^\lambda \Z/ c_i \Z$.

\smallskip
For simplicity we compute an $\alpha$ being an integer without non-trivial rational divisor. So $(\alpha)$ is the 
cube of an ideal if and only if $\No_{k^-/\Q} (\alpha)\in \Q^{\times 3}$. Then the irreducible polynomial 
defining $K$ is given by $P=polcompositum(x^2+x+1, Q)$ 
where $Q=x^6-2\,a\,x^3+(a^2+d\,b^2)$ or $Q=x^6-\,a\,x^3+(a^2+d\,b^2)/4$, where 
$\alpha = a+b\,\sqrt{-d\,}$ or $\Frac{a+b\,\sqrt{-d}}{2}$, respectively.

\smallskip
In all the examples, we recall that the 3-class group of $k^+$ is trivial, and that of $k^-$ is of order $3$ exactely.
Thus, the 3-class group of $K^+$ is trivial and $\Cl_K^-=\Cl_K$; moreover, $\Cl_K^G$ is of order 3.
So the case $n=1$ yields to $a=0, b=1$ ($\Cl_K=\Cl_K^G$), the case $n=3$ yields to $a=1, b=1$, and so on.

\footnotesize
\medskip
\parindent 0pt
allocatemem(1000000000)

$\{$$d=1; while(d<5*10^3, d=d+3; if(core(d)==d, $\par$
D=3*d; if(Mod(D,4)!=1, D=4*D); h=qfbclassno(D); if(Mod(h,3)!=0, $\par$
Dm=-d;  if(Mod(Dm,4)!=1, Dm=4*Dm); hm=qfbclassno(Dm); $\par$
if(Mod(hm,3)==0 \& Mod(hm,9)!=0, $\par$
for(b=1, 10^2, for(a=1, 10^3, if(gcd(a,b)==1, T=2*a; N=a^2+d*b^2; $\par$
if(Mod(-d,4)==1 \& Mod(a*b,2)!=0, T=T/2; N=N/4); $\par$
if(floor(N^{(1/3)})^3-N==0, Q=x^6-T*x^3+N; $\par$
P=polcompositum(x^2+x+1, Q); R=component(P,1); $\par$
H=bnrinit(bnfinit(R, 1),1); F=component(H,5); G=component(F,1); $\par$
if(Mod(G, 3)==0 \& Mod(G, 9)!=0, $\par$
print( "  " );  print(" d= ", d); print("a= ",a); print("b= ",b); $\par$
print("hm = ", hm,"     h = ", h); print(P); $\par$
print("class group : ", F); a=10^3; b=10^2) ))))) ))) $$\}$  

\normalsize
\medskip
The instruction ``if(Mod(G, 3)==0 \& Mod(G, 9)!=0'' must be adapted to the relevant needed structure $(3^{a+1}, 3^a)$.

\medskip
We give below an extract of the numerical examples we have obtained; for a complete table, please see
the Section 5 of:

\smallskip
\url{https://www.researchgate.net/publication/286452614}

\subsection{Case $\Cl_K \simeq \Z/3\Z$ ($a=0$)}\label{inv}
 This implies $\Cl_K = \Cl_K^G\simeq \Z/3\Z$:

\footnotesize
\medskip\medskip
\par$  d= 31
$\par$ a= 1
$,  $ b= 1
$\par$ \# \Cl_{k^-} = 3 $, $   \# \Cl_{k^+} =  1
$\par$ P=x^{12} - 6\,x^{11} + 21\,x^{10} - 52\,x^9 + 99\,x^8 - 144\,x^7 + 179\,x^6 - 186\,x^5 - 33\,x^4 + 268\,x^3 - 87\,x^2 - 24\,x + 64
$\par  $class group : [3, [3]]$\medskip
 
\par$  d= 61
$\par$ a= 8
$,  $\  b= 1
$\par$ \# \Cl_{k^-} = 6  $, $  \# \Cl_{k^+} =  2
$\par$ P=x^{12} - 6\,x^{11} + 21\,x^{10} - 82\,x^9 + 234\,x^8 - 414\,x^7 + 983\,x^6 - 1788\,x^5 - 393\,x^4 - 506\,x^3 + 5394\,x^2 + 4620\,x + 12100
$\par $class group : [12, [6, 2]]$\medskip

{\ \ \ \vdots }

\medskip\smallskip

\par$ d= 913
$\par$ a= 321
$,  $\  b= 4
$\par$ \# \Cl_{k^-}  = 12    $, $   \# \Cl_{k^+} =  8
$\par$ P=x^{12} - 6\,x^{11} + 21\,x^{10} - 1334\,x^9 + 5868\,x^8 - 11682\,x^7 + 661085\,x^6 - 1948290\,x^5 + 702561\,x^4 - 149227072\,x^3 + 227288688\,x^2 + 224655360\,x + 13690872064
$\par $class group : [768, [24, 4, 4, 2]]$\medskip
  
\par$ d= 970
$\par$ a= 563
$,  $\  b= 20
$\par$ \# \Cl_{k^-}  = 12    $, $   \# \Cl_{k^+} =  4
$\par$ P=x^{12} - 6\,x^{11} + 21\,x^{10} - 2302\,x^9 + 10224\,x^8 - 20394\,x^7 + 2701601\,x^6 - 8043702\,x^5 - 2977323\,x^4 - 1568242964\,x^3 + 2378397756\,x^2 + 2373361968\,x + 495396376336
$\par $class group : [600, [30, 10, 2]]$\medskip
  
  \normalsize

\subsection{Case $\Cl_K \simeq  \Z/9\Z \times \Z/3\Z$  ($a=1$)}
${}$

\footnotesize
\medskip

\par$ d= 211
$\par$ a= 17
$,  $\  b= 1
$\par$ \# \Cl_{k^-}  = 3    $, $   \# \Cl_{k^+} =  1
$\par$ P=x^{12} - 6\,x^{11} + 21\,x^{10} - 84\,x^9 + 243\,x^8 - 432\,x^7 + 1037\,x^6 - 1896\,x^5 - 204\,x^4 - 966\,x^3 + 5949\,x^2 + 4905\,x + 11881
$\par $class group : [27, [9, 3]]$\medskip
  
\par$ d= 214
$\par$ a= 89
$,  $\  b= 6
$\par$ \# \Cl_{k^-}  = 6    $, $   \# \Cl_{k^+} =  2
$\par$ P=x^{12} - 6\,x^{11} + 21\,x^{10} - 406\,x^9 + 1692\,x^8 - 3330\,x^7 + 66813\,x^6 - 190530\,x^5 - 45783\,x^4 - 5155600\,x^3 + 8296296\,x^2 + 8156544\,x + 238640704
$\par $class group : [54, [18, 3]]$\medskip
    
{\ \ \ \vdots }

\medskip\smallskip
 
\par$ d= 4531
$\par$ a= 403
$,  $\  b= 5
$\par$ \# \Cl_{k^-}  = 12    $, $   \# \Cl_{k^+} =  2
$\par$ P=x^{12} - 6\,x^{11} + 21\,x^{10} - 856\,x^9 + 3717\,x^8 - 7380\,x^7 + 308855\,x^6 - 904506\,x^5 - 62898\,x^4 - 53921936\,x^3 + 83258895\,x^2 + 82428357\,x + 4694853361
$\par $class group : [1728, [36, 12, 4]]$\medskip
  
\par$ d= 4639
$\par$ a= 361
$,  $\  b= 2
$\par$ \# \Cl_{k^-}  = 51    $, $   \# \Cl_{k^+} =  1
$\par$ P=x^{12} - 6\,x^{11} + 21\,x^{10} - 1494\,x^9 + 6588\,x^8 - 13122\,x^7 + 834341\,x^6 - 2463738\,x^5 + 888141\,x^4 - 212657184\,x^3 + 323349156\,x^2 + 320016960\,x + 21950200336
$\par $class group : [7344, [612, 12]]$\medskip

\normalsize
\subsection{Case $\Cl_K \simeq  \Z/27\Z \times \Z/9\Z$  ($a=2$)}
${}$

\footnotesize
\medskip
  
\par$ d= 1141
$\par$ a= 449
$,  $\  b= 8
$\par$ \# \Cl_{k^-}  = 24    $, $   \# \Cl_{k^+} =  4
$\par$ P=x^{12} - 6\,x^{11} + 21\,x^{10} - 1846\,x^9 + 8172\,x^8 - 16290\,x^7 + 1374653\,x^6 - 4075170\,x^5 + 711057\,x^4 - 487867520\,x^3 + 740542656\,x^2 + 735780864\,x + 74927017984
$\par $class group : [7776, [216, 18, 2]]$\medskip
  
\par$ d= 1174
$\par$ a= 21
$,  $\  b= 5
$\par$ \# \Cl_{k^-}  = 30    $, $   \# \Cl_{k^+} =  2
$\par$ P=x^{12} - 6\,x^{11} + 21\,x^{10} - 134\,x^9 + 468\,x^8 - 882\,x^7 + 62369\,x^6 - 184542\,x^5 - 436569\,x^4 - 1322320\,x^3 + 3316650\,x^2 + 3570000\,x + 885062500
$\par $class group : [2430, [270, 9]]$\medskip
     
{\ \ \ \vdots }

\medskip\smallskip
  
\par$ d= 4087
$\par$ a= 357
$,  $\  b= 8
$\par$ \# \Cl_{k^-}  = 30    $, $   \# \Cl_{k^+} =  2
$\par$ P=x^{12} - 6\,x^{11} + 21\,x^{10} - 1478\,x^9 + 6516\,x^8 - 12978\,x^7 + 1302965\,x^6 - 3870042\,x^5 - 2782815\,x^4 - 543509224\,x^3 + 830485104\,x^2 + 829417344\,x + 150779996416
$\par $class group : [19440, [270, 18, 2, 2]]$\medskip
  
\par$ d= 4567
$\par$ a= 195
$,  $\  b= 1
$\par$ \# \Cl_{k^-}  = 33    $, $   \# \Cl_{k^+} =  7
$\par$ P=x^{12} - 6\,x^{11} + 21\,x^{10} - 440\,x^9 + 1845\,x^8 - 3636\,x^7 + 63557\,x^6 - 179844\,x^5 + 66765\,x^4 - 3988930\,x^3 + 6294021\,x^2 + 6052866\,x + 109286116
$\par $class group : [18711, [2079, 9]]$\medskip
  
\normalsize

\subsection{Case $\Cl_K \simeq  \Z/81\Z \times \Z/27\Z$  ($a=3$)}
${}$

\footnotesize
\medskip

\par$ d= 12058
$\par$ a= 989
$,  $\  b= 26
$\par$ \# \Cl_{k^-}  = 42    $, $   \# \Cl_{k^+} =  2
$\par$ P=x^{12} - 6\,x^{11} + 21\,x^{10} - 4006\,x^9 + 17892\,x^8 - 35730\,x^7 + 22212821\,x^6 - 66531354\,x^5 - 113482743\,x^4 - 35777798264\,x^3 + 54059937672\,x^2 + 54106942656\,x + 83308554531904
$\par $class group : [30618, [1134, 27]]$\medskip
  
\par$ d= 15607
$\par$ a= 534
$,  $\  b= 1
$\par$ \# \Cl_{k^-}  = 39    $, $   \# \Cl_{k^+} =  1
$\par$ P=x^{12} - 6\,x^{11} + 21\,x^{10} - 2186\,x^9 + 9702\,x^8 - 19350\,x^7 + 1764719\,x^6 - 5236188\,x^5 + 2322777\,x^4 - 638361238\,x^3 + 965957748\,x^2 + 958427808\,x + 89817692416
$\par $class group : [28431, [1053, 27]]$\medskip
   
{\ \ \ \vdots }

\medskip\smallskip
 
\par$ d= 45517
$\par$ a= 845
$,  $\  b= 6
$\par$ \# \Cl_{k^-}  = 120    \# \Cl_{k^+} =  4
$\par$ P=x^{12} - 6\,x^{11} + 21\,x^{10} - 3430\,x^9 + 15300\,x^8 - 30546\,x^7 + 7597005\,x^6 - 22699458\,x^5 - 18168075\,x^4 - 7877764840\,x^3 + 11909686236\,x^2 + 11905200672\,x + 5526956498704
$\par $class group : [174960, [3240, 54]]$\medskip
  
\par$ d= 47194
$\par$ a= 293
$,  $\  b= 2
$\par$ \# \Cl_{k^-}  = 120    \# \Cl_{k^+} =  4
$\par$ P=x^{12} - 6\,x^{11} + 21\,x^{10} - 1222\,x^9 + 5364\,x^8 - 10674\,x^7 + 905093\,x^6 - 2683338\,x^5 - 2064183\,x^4 - 313267016\,x^3 + 480721224\,x^2 + 480118080\,x + 75097921600
$\par $class group : [699840, [3240, 54, 2, 2]]$

\normalsize

\subsection*{Acknowledgments} I thank Christian Maire for telling me about difficulties with the techniques developed in \cite{W}, 
Thong Nguyen Quang Do and Jean-Fran\c cois Jaulent for similar conversations and advices about it.

\end{document}